\documentclass[a4paper,12pt]{amsart}
\usepackage{amssymb}
\usepackage{amsmath}
\usepackage{graphics}
\usepackage{soul}
\usepackage{xcolor}

\usepackage[dvips]{graphicx}
\newtheorem{thm}{Theorem}

\newtheorem{cor}{Corollary}
\newtheorem{lemma}[thm]{Lemma}

\def\C{{\mathbb C}}

\def\D{{\mathbb D}}

\def\be{\begin{equation}}
\def\ee{\end{equation}}

\makeatletter

\begin{document}

\title[Two-point distortion theorems for harmonic mappings]{Two-point distortion theorems for harmonic mappings}

\author{V. Bravo\and R. Hern\'andez \and O. Venegas }
\thanks{The
authors were partially supported by Fondecyt Grants \# 1190756.
\endgraf  {\sl Key words:} {Two-point distortion, harmonic mappings, univalence criterion.}
\endgraf {\sl 2020 AMS Subject Classification}. Primary: 30C45, 30C99; \,
%32H02, 32A17; \,
Secondary: 31C05.}

\address{Facultad de Ingenier\'ia y Ciencias\\
Universidad Adolfo Ib\'a\~nez\\
Av. Padre Hurtado 750, Vi\~na del Mar, Chile.}
\email{victor.bravo.g@uai.cl}

\address{Facultad de Ingenier\'ia y Ciencias\\
Universidad Adolfo Ib\'a\~nez\\
Av. Padre Hurtado 750, Vi\~na del Mar, Chile.}
\email{rodrigo.hernandez@uai.cl}

\address{Departamento de Ciencias Matem\'{a}ticas y F\'{\i}sicas. Facultad de Ingenier\'{\i}a\\ Universidad Cat\'olica de
Temuco, Chile.} \email{ovenegas@uct.cl}

%\address{}\email{}
%\address{}\email{}

\begin{abstract} We establish two-point distortion theorems for sense-preserving planar harmonic mappings $f=h+\overline{g}$ which satisfies the univalence criteria in the unit disc such that, Becker's and Nehari`s harmonic version. In addition, we find the sharp two-point distortion theorem when $h$ is a convex function, and normalized mappings such that $h(\D)$ is a $c$-linearly connected domain. To do this, we use the order of this family.
\end{abstract}

\maketitle

\section{Introduction}

To a large extent, the two-point distortion theorems provide us quantitative information on ``how injective" a conformal function is, in the sense of estimating the distance between the images of any two points, respect to the distance (Euclidean or Hyperbolic) of these two points. That is, we seek lower and upper bounds for the distance between $f(a)$ and $f(b)$ for all $a$ and $b$ in the unit disc $\D$ and that involves in some sense  $\rho(a,b)=\mid (a-b)/(1-\overline{a}b)\mid$ and the derivative of the function $f$. The first two-point distortion theorem was introduced by C. Blatter in \cite{B} and proved that if $f$ is a conformal mapping, then $$\mid f(a)-f(b)\mid^2\geq \frac{\sinh^2(2d(a,b))}{8(\cosh(4d(a,b)))}\left(R(a)^2+R(b)^2\right).$$ Here $d(a,b)=\tanh^{-1}(\rho(a,b))$, and $R(z)=(1-\vert z\vert^2)\vert f'(z)\vert $. Moreover, for different subclasses of univalent functions, such as convex, starlike, etc., there are two-point distortion theorems in which the bounds are more accurate. In particular, if $f$ conformally maps the disk onto a convex domain, Kim and Minda in \cite{KM} proved the following result:\\ 

\noindent {\bf Theorem A.} Let $f:\D\to \C$ be a conformal convex mapping, then for all $p>1$ and all $a$, $b$ in $\D$ we have that \begin{equation}
\vert f(a)-f(b)\vert \geq \frac{\sinh(d(a,b))}{2(\cosh(pd(a,b)))^{1/p}}\left(R(a)^p+R(b)^p\right)^{1/p}.
\end{equation}

On the other hand, univalence criteria involving both the pre-Schwarzian and Schwarzian derivatives of locally univalent functions defined in $\D$ are well known, which also provide their own two-point distortion theorems. Let $f$ be a locally univalent mapping, its \textit{Schwarzian} derivative is defined by:

\begin{equation}\label{sf} Sf=\left(\frac{f''}{f'}\right)'-\frac12\left(\frac{f''}{f'}\right)^2=Pf'-\frac12Pf^2,\end{equation} here $Pf=f''/f'$ is named \textit{Pre-Schwarzian} derivative of $f$. Nehari in \cite{N49} showed that $\|Sf\|=\sup\{\vert Sf(z)\vert (1-\vert z\vert ^2)^2:z\in\D\}\leq 2$ implies that $f$ is univalent in $\D$. To see more univalence criteria that involves the Schwarzian derivative, we can see the $p$-criterion due by Nehari in \cite{N54}. In this context, Chuaqui and Pommerenke \cite{ChP} gave a quantitative version of Nehari’s theorem by showing that the condition $\|Sf\|\leq 2$ implies that $f$ has the two-point distortion property \begin{equation}\label{nehari-two-point}\vert f(a)-f(b)\vert \geq d(a,b)\sqrt{R(a)R(b)}.\end{equation}  This result was extended to the $p$-criterion by Ma, Mejía and Minda in \cite{MMM}.\\

This paper aims to obtain two-point distortion theorems associated with different univalence criteria for complex harmonic functions defined in $\D$. These injectivity criteria are associated with both the Pre-Schwarzian and Schwarzian derivatives, but also to harmonic functions whose analytic part is convex, see \cite{ChH07}. The manuscript is organized as follows: In Section 2 deals with the preliminaries on complex harmonic mappings and includes the above univalence criteria, called Theorem B, Theorem C, and Theorem D. Subsections 3.1 and 3.2 are devoted to the two-point distortion theorems associated with the Becker and Nehari criteria respectively, but in their complex harmonic versions. Special attention will be paid in subsection 3.3 since we not only deal with the case when the analytic part is convex but also when it conformally maps the disk into a \textit{linearly connected} domain which means geometrically that it has no inward cusps. Indeed, we prove that these functions form a linearly invariant family which follows the two-point distortion theorem found in this manuscript.

\section{Preliminaries}

A complex-valued harmonic function $f$ in a simply connected domain $\Omega$ are those who satisfies that $\Delta f=\partial^2 f/\partial z\partial \overline z=0$, which is equivalent to $f=u+iv$ where $u$ and $v$ are real harmonic mappings defined in $\Omega$. A straightforward calculation give that this functions has a canonical representation $f=h+\overline g$, where $h$ and $g$ are analytic functions in $\Omega$, that is unique up to an additive constant. When $\Omega=\D$ it is convenient to choose the additive constant so that $g(0)=0$. The representation $f=h+\overline g$ is then unique and is called the canonical representation of $f$. A result of Lewy \cite{Lewy} states that $f$ is locally univalent if and only if its Jacobian $J_f=\vert h'\vert^2-\vert g'\vert ^2$ does not vanish in $\Omega$. Thus, harmonic mappings are either sense-preserving or sense-reversing depending on the conditions $J_f >0$ or $J_f <0$ throughout the domain $\Omega$ where $f$ is locally univalent, respectively. Since $J_f>0$  if and only if $J_{\overline f}<0,$ throughout this work we will consider sense-preserving mappings in $\D.$ In this case the analytic part $h$ is locally univalent in $\D$ since $h'\neq0$, and the second complex dilatation of $f$, $\omega=g'/h'$, is an analytic function in $\D$ with $\vert \omega\vert<1$. The reader can be found in \cite{Dur} a elegant reference for this topic.\\

Duren, Hamada \& Kohr in \cite{DHK} showed, among other results, that $f=h+\overline g$ univalent harmonic mapping and normalized to $h(0)=g(0)=0$, and $h'(0)=1$, and let $\alpha$ the order of this family, which is the supremum of the absolute value of the second Taylor coefficient of $h$, then $$\vert f(a)-f(b)\vert\geq\frac{1}{2\alpha}(1-\exp(-2\alpha d(a,b)))\max\{R(a),R(b)\},$$ where $R(z)=(1-\vert z\vert ^2)(\vert h'(z)\vert-\vert g'(z)\vert).$ Moreover, they also proved that $$\vert f(a)-f(b)\vert\leq\frac{1}{2\alpha}(\exp(2\alpha d(a,b)-1))\min\{Q(a),Q(b)\},$$ where $Q(z)=(1-\vert z\vert ^2)(\vert h'(z)\vert+\vert g'(z)\vert)$. Which represent the harmonic version of Blatter's Theorem. Note that in the analytic case, $\alpha=2$, however, in the harmonic case, the size of $\alpha$ is unknown.\\

Hern\'andez and Mart\'in \cite{RH-MJ}, defined the harmonic pre-Schwarzian and Schwarzian derivatives for sense-preserving harmonic mappings
$f=h+\overline g$. Using those definitions, they generalized different results
regarding analytic functions to the harmonic case
\cite{RH-MJ,rhmje,HM-QC,HM-arch}. The pre-Schwarzian and Schwarzian derivatives of a
sense-preserving harmonic function $f$ are corresponding defined by
\begin{equation*}\label{Pf}
P_f=\frac{h''}{h'}-
\frac{\bar{\omega}\omega'}{1-\vert\omega\vert^2}=Ph-\frac{\bar{\omega}\omega'}{1-\vert\omega\vert^2}=\frac{\partial}{\partial
z}\log(J_f),
\end{equation*}

\begin{equation*}\label{Sf}
S_f=\frac{\partial}{\partial z}P_f-\frac12(P_f)^2=Sh+\frac{\overline \omega}{1-\vert\omega\vert ^2}\left(\omega'\frac{h''}{h'}-\omega''\right)-\frac32\left(\frac{\omega'\overline \omega}{1-\vert \omega\vert^2}\right)^2,
\end{equation*} where $Ph$ and $Sh$ are the standard Pre-Schwarzian and Schwarzian derivative respectively, which is given by equation (\ref{sf}). It is easy to see that if $f$ is analytic ($f=h$ and
$\omega=0$) then $P_f=h''/h'$ and $S_f=Sh$, recovering the classical definition of this
operators. In \cite{RH-MJ} the authors proved that $S_f=0$ if and only if $f$ is a M\"obius harmonic transformation, which is given by $f=h+\alpha \overline h$ with $\alpha\in\D$ and $h$ is a M\"obius analytic function. To motivate the two-point distortion theorems associated to univalence criteria, those harmonic M\"obius mappings satisfies that  $$\vert f(a)-f(b)\vert=\sqrt{\vert h'(a)\vert \vert h'(b)\vert}\vert a-b\vert \vert1+\lambda \alpha\vert,\quad \lambda=\frac{\overline{h(a)-h(b)}}{h(a)-h(b)}\,,\quad  a\neq b.$$ Thus, we can re-write this equation to get
$$\vert f(a)-f(b)\vert=\sqrt{R_h(a)R_h(b)}\sinh(d(a,b))\vert 1+\lambda\alpha\vert,\quad R_h(z)=\vert h'(z)\vert (1-\vert z\vert^2).$$ In addition, for any $a\in\D$, the authors also proved that
$P_{f+\overline {af}}=P_f$ which is the key to prove the extension of Becker's criterion of univalence for harmonic mappings. This is contained in the following theorem:\\

\noindent {\bf Theorem B.} Let $f=h+\overline{g}$  be a sense-preserving
harmonic function in the unit disc $\D$ with dilatation $\omega$. If for all
$z\in \D$ $$
(1-\vert z\vert ^2)\vert zP_f(z)\vert+\frac{\vert z\omega'(z)\vert(1-\vert z\vert ^2)}{1-\vert \omega(z)\vert ^2} \leq 1,$$
then $f$ is univalent. The constant 1 is sharp.

The classical Becker's criterion can be found in \cite{B72} and the extension to harmonic mapping appear in \cite{RH-MJ}.\\

Observe that using the Schwarz-Pick lemma, $\omega:\D\to\D$ satisfies that \begin{equation}\label{omega*}
\|\omega^*\|=\sup_{z\in\D}\dfrac{\vert\omega'(z)\vert(1-\vert z\vert ^2)}{1-\vert \omega(z)\vert ^2}\leq1.\end{equation}
In addition, we consider $\|\omega\|=\sup\{\vert \omega(z)\vert:\,z\in\D\}=\|\omega\|_\infty$.
In \cite{HM-arch} the authors proved, using that for all $a\in\D$, $S_{f+\overline{af}}=S_f$, the following generalization of Nehari's classical univalence criterion:\\

\noindent{\bf Theorem C.} Let $f=h+\overline{g}$ be a sense-preserving harmonic function in the unit disc $\D$ with dilatation $\omega$. Then, there exists $\varepsilon>0$ such that 
$$(1-\vert z\vert^2)^2\vert S_f(z)\vert 
\leq \varepsilon$$ implies that $f$ is univalent. 

Obviously, the classical result is obtained taking $\omega=0$, but in this case, $\varepsilon=2$. Unfortunately, the value of $\varepsilon$ is still unknown for the planar harmonic mappings setting. From the proof of this theorem, we can see that this value $\varepsilon$ is small enough to $\|Sh\|\leq 2$, i.e. $h$ satisfies the Nehari's classical univalence criterion. Moreover, since $S_f$ is invariant under pre-composition of affine mapping, we have that $h+ag$ satisfies the Nehari's criterion for all $a\in \overline{\D}$.\\

The last criterion that we will consider in this manuscript appears in \cite{ChH07} and says the following:\\

\noindent{\bf Theorem D.} Let $f=h+\overline{g}$ be a sense-preserving harmonic function in the unit disc $\D$ with dilatation $\omega$. If $h$ is a convex mapping, then $f$ is univalent. In fact, when $h(\D)$ is $c-$linearly connected domain, if $\|\omega\|_\infty<1/c$, then $f$ is also univalent in the unit disc.\\

Thus, the principal goal of this notes is give to measure of the univalence of functions that satisfies this inyectivity criteria of the form of two-point distortion theorems. Generalizing the classical results for holomorphic mappings.

\section{Results}

Troughs the section, we set $\varphi$ as a locally univalent analytic, and $f=h+\overline g$ be a sense-preserving harmonic mapping with dilatation $\omega$ defined in the unit disc. Recall that the fact that $f$ is sense-preserving, implies that $h'\neq 0$ and $\omega(\D)\subset\D$. Additionally, we call $R_\varphi(z)=(1-\vert z^2\vert)\vert \varphi'(z)\vert$ and   $R(z)=(1-\vert z\vert ^2)(\vert h'(z)\vert-\vert g'(z)\vert)$, and $Q(z)=(1-\vert z\vert ^2)(\vert h'(z)\vert+\vert g'(z)\vert)$ its harmonic version of this quantity.

\subsection{The Becker's criterion case}
Let $\varphi$ be a normalized locally univalent analytic function defined in $\D$. Its order is given by $$\mbox{ord}\langle\varphi\rangle=\sup_{z\in\D}\left\vert \frac12\frac{\varphi''}{\varphi'}(z)(1-\vert z\vert^2)-\overline z\right\vert=\sup_{a\in \D}\left\{\frac12\vert \varphi_a''(0)\vert: \varphi_a(z)=\dfrac{\varphi\left(\dfrac{z+a}{1-\overline{a}z}\right)-\varphi(a)}{(1-\vert a\vert ^2)\varphi'(a)} \right\},$$ which is basically the supremum of second Taylor coefficient of $\varphi_a$, for all $a\in\D$.  A well-known result (see \cite{POM1}) asserts that if $\varphi$ is  univalent mapping defined in the unit disk, normalized to $\varphi(0)=0$ and $\varphi'(0)=1$, and its order is $\alpha,$ then \begin{equation}\label{crecimiento orden} \frac{1}{2\alpha}\left(\left(\frac{1+\vert z\vert}{1-\vert z\vert}\right)^\alpha-1\right)\geq \vert\varphi(z)\vert\geq \frac{1}{2\alpha}\left(1-\left(\frac{1-\vert z\vert}{1+\vert z\vert}\right)^\alpha\right), \quad \quad z\in\D.
\end{equation} Thus, the following Theorem is an application of this when $\varphi$ satisfies the classical univalence criterion due by Becker in \cite{B}. Actually, we haven't been able to find any reference on this, so we include a proof.

\begin{thm}
Let $\varphi:\D\to \C$ be a locally univalent  normalized as before such that $ \vert P\varphi(z)\vert(1-\vert z\vert^2)\leq 1$. Then for all $a,b\in \D$ we have that  \begin{equation}\label{Becker-two-point}\frac{\exp(3d(a,b))-1}{3}\sqrt{R_\varphi(a)R_\varphi(b)}\geq\vert\varphi(a)-\varphi(b)\vert\geq\dfrac{1-\exp(-3d(a,b))}{3}\sqrt{R_\varphi(a)R_\varphi(b)}.\end{equation}
\end{thm}

\begin{proof}
We note that $$\alpha=\mbox{ord}\langle\varphi \rangle=\sup_{z\in\D}\left\vert\dfrac{1}{2}(1-\vert z\vert^2)\dfrac{\varphi''}{\varphi'}(z)-\overline{z}\right\vert\leq \dfrac{1}{2}+\vert z\vert\leq \dfrac{3}{2},$$ then $\varphi_a$ satisfies inequality (\ref{crecimiento orden}).
Taking $b=\dfrac{z+a}{1-\overline{a}z}$ which is equivalent to $z=\dfrac{b-a}{1-\overline{a}b}$ we obtain $$\left\vert\dfrac{\varphi(b)-\varphi(a)}{(1-\vert a\vert^2)\vert \varphi'(a)\vert}\right\vert\geq \dfrac{1}{2\alpha}\left[1-\left(\dfrac{1-\left\vert\dfrac{b-a}{1-\overline{a}b}\right\vert}{1+\left\vert\dfrac{b-a}{1-\overline{a}b}\right\vert}\right)^\alpha\right]\geq \dfrac{1}{3}\left[1-\left(\dfrac{1-\left\vert\dfrac{b-a}{1-\overline{a}b}\right\vert}{1+\left\vert\dfrac{b-a}{1-\overline{a}b}\right\vert}\right)^{3/2}\right].$$
Since $\rho(a,b)=\left\vert\dfrac{b-a}{1-\overline{a}b}\right\vert$ from the last inequality we obtain $$\vert\varphi(a)-\varphi(b)\vert^2\geq\dfrac{(1-\vert a\vert^2)(1-\vert b\vert^2)\vert\varphi'(a)\vert \vert\varphi'(b)\vert}{9}\left[1-\left(\dfrac{1-\rho(a,b)}{1+\rho(a,b)}\right)^{3/2}\right]^2.$$ But, $$d(a,b)=\frac12\log\left(\dfrac{1+\rho(a,b)}{1-\rho(a,b)}\right),$$ therefore taking square root we obtain the lower bound. 

Since the real function $(\xi^\alpha-1)/\alpha$ is growing in $\alpha$ for any $\xi>1$, using $\alpha\leq3/2$ and \cite[Staz 1.1]{POM1} again, we have that $$\vert\varphi(a)-\varphi(b)\vert\leq \frac{\sqrt{(1-\vert a\vert^2)(1-\vert b\vert^2)\vert\varphi'(a)\vert \vert\varphi'(b)\vert}}{3}\left[\left(\dfrac{1+\rho(a,b)}{1-\rho(a,b)}\right)^{3/2}-1\right],$$ which ends the proof.
\end{proof}

Note that, under the hypothesis of Theorem B, the analytic part of $f=h+\overline{g}$ satisfies that $\vert Ph(z)\vert(1-\vert z\vert^2)\leq 1$ then satisfies the Becker's univalence criterion for analytic mappings in $\D$. Since $P_{f+\lambda \overline{f}}=P_f$ then $\varphi_\lambda=h+\lambda g$ satisfies the Becker's univalence criterion for all $\lambda \in\D$, moreover, taking limits when $\vert\lambda\vert\to 1$, we can asserts that Becker's criterion holds for $\vert\lambda\vert\leq 1$, which is the key for the following theorem:

\begin{thm} Let $f=h+\overline{g}$ be a sense-preserving harmonic mapping defined in $\D$, such that satisfies the hypothesis of Theorem B, then 
$$\frac{1}{3}(\exp(3 d(a,b))-1)\sqrt{Q(a)Q(b)}\geq \vert f(a)-f(b)\vert\geq \frac{1}{3}(1-\exp(-3 d(a,b)))\sqrt{R(a)R(b)}.$$
\end{thm}

\begin{proof} For $a$ and $b$ in $\D$, it follows that $f(a)-f(b)=h(a)-h(b)+\overline{(g(a)-g(b))}$, thus, if $g(a)\neq g(b)$ there exists an unimodular constant $\lambda$ such that $f(a)-f(b)=\varphi_\lambda(a)-\varphi_\lambda(b)$. In the case when $g(a)=g(b)$ we can consider $\lambda=0$. In any case, $\varphi_\lambda$ satisfies the Becker's criterion, thus we can apply inequality (\ref{Becker-two-point}) and the fact that $\vert h'\vert+\vert g'\vert\geq \vert\varphi'_\lambda\vert\geq \vert h'\vert-\vert g'\vert$ to obtain that $$\vert f(a)-f(b)\vert^2\leq \frac13\left[\left(\frac{1+\rho(a,b)}{1-\rho(a,b)}\right)^{3/2}-1\right]Q(a)Q(b),$$ and $$\vert f(a)-f(b)\vert^2\geq \frac13\left[1-\left(\frac{1-\rho(a,b)}{1+\rho(a,b)}\right)^{3/2}\right]R(a)R(b).$$
\end{proof}

\subsection{The Nehari's criterion case}

Let $f=h+\overline{g}$ be a sense-preserving harmonic mapping. In \cite{RH-MJ} the authors shows that there exists $\varepsilon>0$ such that if $\|S_f\|=\sup\{(1-\vert z\vert^2)^2\vert S_f(z)\vert:z\in\D\}\leq \varepsilon$, then $f$ is univalent in $\D$. Moreover, is deduce that there exists a constant $k>0$ such that $$\|Sh\|\leq \varepsilon+kW_\varepsilon+\frac32W_\varepsilon^2,$$ where $W_\varepsilon=\sup\{\|\omega^\ast\|:\omega \in A_\varepsilon\}$. We say that an analytic mapping $\omega$ belongs to $A_\varepsilon$ if and only there exist $f=h+\overline g$ with dilatation $\omega$ and $\|S_f\|\leq \varepsilon$. In the same paper, has been proved that $W_\varepsilon\to0$ when $\varepsilon\to 0$. We can assume that $\varepsilon$ is given by $$\varepsilon=\sup\{\delta:\delta+kW_\delta+\frac32W_\delta^2\leq 2\},$$ thus the Theorem C asserts that $\|S_f\|\leq \varepsilon$ is a univalence criterion for sense-preserving harmonic mappings. From the proof of Theorem C it follows that for $\lambda\in\D$, the locally univalent analytic mapping $\varphi_\lambda=h+\lambda g$,  satisfies that $\|S\varphi_\lambda\|
\leq\varepsilon$ since $f_\lambda=f+\lambda\overline f$ has the same Schwarzian derivative. Thus we can asserts the following lemma:

\begin{lemma} Let $f=h+\overline g$ be a sense-preserving harmonic mapping defined in $\D$ such that $\|S_f\|\leq \varepsilon$, then for any $\lambda\in\D$, the corresponding mapping $\varphi_\lambda$ satisfies that $\|S\varphi_\lambda\|\leq 2$.
\end{lemma}

In \cite[Thm. 1.1 part a)]{MMM} the authors proved that if $\|S\varphi\|\leq2t$ with $t\in[0,1]$ then \begin{equation}\label{nehari-t}\vert\varphi(a)-\varphi(b)\vert\leq \sqrt{\dfrac{R_\varphi(a)R_\varphi(b)}{1+t}}\sinh\left({\sqrt{1+t}\,d(a,b)}\right),\end{equation} which is, together \cite[Thm. 2]{ChP} the key inequalities for the following theorem.

\begin{thm} Let $f=h+\overline g$ be a sense-preserving harmonic mapping defined in $\D$. If $\|S_f\|\leq \varepsilon$ where $\varepsilon$ as in Theorem B, then \begin{equation*}\sqrt{\frac{Q(a)Q(b)}{2}}\sinh\left(\sqrt{2}\,d(a,b)\right)\geq\vert f(a)-f(b)\vert\geq d(a,b)\sqrt{R(a)R(b)}.
\end{equation*} 
\end{thm}

\begin{proof} For any $a$ and $b$ in the unit disk, it follows that $f(a)-f(b)=h(a)-h(b)+\overline{(g(a)-g(b))}$, thus, if $g(a)\neq g(b)$ there exists an unimodular constant $\lambda$ such that $f(a)-f(b)=\varphi_\lambda(a)-\varphi_\lambda(b)$. In the case when $g(a)=g(b)$ we can consider $\lambda=0$. But, using Theorem 2 in \cite{ChP} (equation \ref{nehari-two-point}) and inequality (\ref{nehari-t}) we have that $$\frac{R_{\varphi_\lambda}(a)R_{\varphi_\lambda}(b)}{2}\sinh^2\left(\sqrt{2}\,d(a,b)\right)\geq\vert \varphi_\lambda(a)-\varphi_\lambda(b)\vert^2\geq d(a,b)^2R_{\varphi_\lambda}(a)R_{\varphi_\lambda}(b),$$ but $\vert h'\vert+\vert g'\vert\geq\vert \varphi_\lambda'\vert\geq \vert h'\vert-\vert g'\vert$ which completes the proof.
\end{proof}

Note that if $f$ is analytic, which is the case when $\omega=0$, then $W_\varepsilon=0$ and $\varepsilon=2$, thus, the last theorem is a generalization of \cite[Thm. 2]{ChP}.

\subsection{When the analytic part is a convex function}

We known that  $f=h+\overline{g}$ is univalent in the unit disk when $h$ is a convex mapping (by \cite[Thm. 1]{ChH07}). Moreover, the result can be extended to the case when $h$ maps conformally $\D$ onto a linearly connected domain, which is the matter in sub-section 3.1.1.

\begin{thm}
Let $f=h+\overline g$ be a sense-preserving harmonic mapping defined in $\D$ such that $h$ is a convex function. Then for all $a,b\in \D$ $$\vert f(a)-f(b)\vert\leq (1+\|\omega\|_\infty)\left\vert\dfrac{a-b}{1-\overline{a}b}\right\vert\left(\dfrac{R_h(a)+R_h(b)}{2}\right),$$ and $$\vert f(a)-f(b)\vert\geq (1-\|\omega\|_\infty)\left\vert\dfrac{a-b}{1-\overline{a}b}\right\vert\left(\dfrac{R_h(a)+R_h(b)}{2}\right).$$ 
\end{thm}

\begin{proof}
We note that 
\begin{eqnarray}\vert h(a)-h(b)\vert+\vert g(a)-g(b)\vert \geq\vert f(a)-f(b)\vert \geq \vert h(a)-h(b)\vert-\vert g(a)-g(b)\vert.\label{eq1}\end{eqnarray} 
Considering the curve $\gamma=h^{-1}(R),$ where $R$ is the segment joining $h(a)$ and $h(b)$, then for $\zeta=h^{-1}(w)$ we have that $d\zeta=\dfrac{1}{h'(h^{-1}(w))}dw$ and  
$$
\vert g(a)-g(b)\vert=\left\vert\int_{\gamma}g'(\zeta)\, d\zeta\right\vert
=  \left\vert\int_R\dfrac{g'(\zeta)}{h'(\zeta)}dw\right\vert
\leq  \|\omega\|_\infty\int_R\vert dw\vert
= \|\omega\|_\infty \vert h(a)-h(b)\vert.
$$

Hence, substituting this last inequality in (\ref{eq1}) we obtain \begin{equation*}\label{desigualdad-h-convexa}(1+\|\omega\|_\infty)\vert h(a)-h(b)\vert\geq\vert f(a)-f(b)\vert\geq (1-\|\omega\|_\infty)\vert h(a)-h(b)\vert.\end{equation*}
Therefore, using Theorem A follows the result.
\end{proof}

%Finalmente por el teorema A se tiene que
%$$|f(a)-f(b)|\geq (1-\|\omega\|_\infty)\left|\dfrac{a-b}{1-\overline{a}b}\right|\left(\dfrac{R_h(a)+R_h(b)}{2}\right).$$

\subsubsection{When the analytic part $h$ is $c$-linearly connected}

We will pay special attention to the case where $h(\D)$ is a linearly connected domain, which means that for any $a$ and $b$ in $\D$ there exists a curve $\gamma$ that joined $h(a)$ with $h(b)$, such that its length satisfies that \begin{equation}\ell_\gamma\leq c\vert h(a)-h(b)\vert.
\label{gamma}\end{equation} In this case, we say that $f$ is a $c$-linearly connected. Observed that $c\geq 1$ and $c=1$ if and only if $h(\D)$ is a convex domain. Was proved in \cite{ChH07} that if $h(\D)$ is a linearly connected domain with constant $c$ and $\vert\omega\vert<1/c$ then $f$ is univalent in the unit disc. 

\begin{thm} Let $\varphi:\D\to\C$ be a conformal and $c$-linearly connected mapping defined in $\D$, then $$\varphi_a(z)=\dfrac{\varphi\left(\dfrac{z+a}{1+\overline{a}z}\right)-\varphi(a)}{(1-\vert a\vert^2)\varphi'(a)}$$ is a $c$-linearly connected mapping.  
\end{thm}

Moreover, it is known that the set $\mathcal F$, of normalized functions by $\varphi(0)=0$ and $\varphi'(0)=1$ such that for any $a\in\D$, $\varphi_a$ belongs to $\mathcal F$ is called \textit{linearly invariant family}, and has been largely studied by Ch. Pommerenke, for instance \cite{POM1}. One of the most famous families is the conformal, called $S$. It is well-known that the order of $S$ is less or equal than 2. By the last Theorem, given $c\geq1$ the set $K_c=\{\varphi\in S:\varphi(\D)\,\mbox{is $c$-linearly connected}\}$ is a linearly invariant family. Recall that $K_1$ is the classical family of convex mappings $K$ (see \cite{Dur83}). Let $\beta$ the order of $K_c$, which depends of $c$ and tends to 2 when $c$ goes to $\infty$ and since $K_c\subset S$, then $\beta< 2$.

\begin{thm}  Let $f=h+\overline g$ be a sense-preserving harmonic mapping defined in $\D$ such that $h$ is a $c$-linearly connected mapping and its dilatation satisfies that $\|\omega\|_\infty<1/c$, then for any $a$ and $b$ in $\D$ $$\vert f(a)-f(b)\vert\geq (1-c\|\omega\|_\infty)\frac{1}{2\beta}(1-\exp(-2\beta d(a,b))\sqrt{R_h(a)R_h(b)},$$ and $$\vert f(a)-f(b)\vert\leq (1+c\|\omega\|_\infty)\frac{1}{2\beta}(\exp(2\beta d(a,b)-1)\sqrt{R_h(a)R_h(b)},$$ where $\beta$ is the order of $h$. 
\end{thm}

\begin{proof} For any $a$ and $b$ in $\D$ exists a curve $\gamma\subset h(\D)$ such that its large $\ell_\gamma$, satisfies (\ref{gamma}). Let $\Gamma=h^{-1}(\gamma)\subset\D$, thus, considering $z=h^{-1}(\zeta)$ with $\zeta\in\gamma$, we have that $$\vert g(a)-g(b)\vert=\left\vert\int_\Gamma g'(z)dz\right\vert\leq\int_{\gamma}\frac{\vert g'(z)\vert}{\vert h'(z)\vert}\vert d\zeta\vert\leq\|\omega\|_\infty\ell_\gamma\leq c\|\omega\|_\infty \vert h(a)-h(b)\vert.$$ Let $\beta$ the order of $h$, which is less than 2. Since $\vert f(a)-f(b)\vert=\vert h(a)-h(b)+\overline{(g(a)-g(b))}\vert$, \cite[Staz 1.1]{POM1}, and triangle inequality, the proof is complete.

\end{proof}

Moreover, for any $\lambda$ with modulus 1, the corresponding analytic function $\varphi_\lambda$ is also a linearly connected mapping with constant $1-c\|\omega\|_\infty$. Since $$\frac{\varphi_\lambda'' }{\varphi_\lambda'}=\frac{h''}{h'}+\frac{\lambda\omega'}{1+\lambda \omega},$$ the order of $\varphi_\lambda$ named $\beta_\lambda$ satisfies that $\beta_\lambda=\min\{2,\beta+\|\omega^\ast\|\},$ where $\|\omega^\ast\|$ given by equation (\ref{omega*}). Thus, we can assert the following corollary.

\begin{cor} Let $f$ as previous theorem. Then, for any $a$ and $b$ in $\D$ $$\frac{1}{2\beta_\lambda}(\exp(2\beta_\lambda d(a,b))-1)\sqrt{R(a)R(b)}\geq |f(a)-f(b)|\geq \frac{1}{2\beta_\lambda}(1-\exp(-2\beta_\lambda d(a,b)))\sqrt{R(a)R(b)}.$$

\end{cor}


\begin{thebibliography}{99}

\bibitem{B72} Becker J., L\"{o}wnersche Differentialgleichung und quasi-konform forttsetzbare schlichte funktionen, J. Reine. Angew. Math., 225, 23--43 (1972).\\

\bibitem{B} Blatter C., Ein verzerrungssatz f\"ur schlichte funktionen, Comment. Math. Helv. 53, 651–659 (1978).\\

\bibitem{ChH07} Chuaqui M. and Hernández R., Univalent harmonic mappings and linearly connected domains, JMAA 332, 1189–1194 (2007).\\

\bibitem{ChP} Chuaqui M. and Pommerenke Ch., Characteristic properties of Nehari functions, Pacific J. Math. 188(1), 83–94 (1999).\\

\bibitem{Dur} Duren P. L., Harmonic mappings in the plane, Cambridge University Press, (2004).\\

\bibitem{Dur83} Duren P. L., Univalent functions, Springer-Verlag, New York, (1983).\\

\bibitem{DHK} Duren P., Hamada H. and Kohr G., Two-point distortion theorems for harmonic and pluriharmonic mappings, Trans. Amer. Math. Soc. 363 (12), 6197--6218 (2011).\\


\bibitem{RH-MJ} Hern\'andez R. and Mart\'in M. J., Pre-Schwarzian and Schwarzian derivatives of Harmonic Mappings, J. Geom. Anal., 25(1), 64--91 (2015).\\

\bibitem{rhmje} Hern\'andez R. and Mart\'in M. J., Stable geometric properties of analytic and harmonic functions, Math. Proc. Cambridge Philos. Soc., 155(2), 343--359 (2013).\\

\bibitem{HM-QC} Hern\'andez R. and Mart\'in M. J., Quasi-conformal extensions of harmonic mappings in the plane, Ann. Acad. Sci. Fenn. Ser. A. I Math, 38, 617--630 (2013).\\

\bibitem{HM-arch} Hern\'andez R., and Mart\'in M. J., Criteria for univalence and quasiconformal extension of harmonic mappings in terms of the Schwarzian derivative, Arch. Math.  104(1), 53--59 (2015).\\

\bibitem{KM} Kim, S. and Minda, D., Two-point distortion theorems for univalent functions, Pac. J. Math.  163(1), 137-157 (1994).\\

\bibitem{Lewy} Lewy H., On the non-vanishing of the Jacobian in certain one-to-one mappings, Bull. Amer. Math. Soc. 42, 689--692 (1936).\\

\bibitem{MMM} Ma W., Mejía D. and Minda D., Two-point distortion for Nehari functions, Complex Anal. Oper. Theory. 8, 213–225 (2014).\\

\bibitem{N49} Nehari, Z., The Schwarzian derivative and schlicht functions. Bull. Am. Math. Soc. 55, 545–551 (1949).\\ 

\bibitem{N54} Nehari Z., Some criteria of univalence, Proc. Amer. Math. Soc. 5, 700–704 (1954).\\

\bibitem{POM1}Pommerenke Ch., Linear-invariante familien analytischer funktionen {I}, Math. Ann. 155, 108--154 (1964).

\end{thebibliography}
\end{document}